\documentclass[12pt, a4paper]{amsart}
\usepackage[utf8]{inputenc}
\usepackage{amsfonts}
\usepackage{amsthm}
\usepackage{tikz}
\usepackage{soul}
\usepackage[normalem]{ulem}

\def\A{\mathcal{A}}
\def\e{\varepsilon} 
\def\N{\mathbb{N}}
 
\def\S1{\mathbb{S}^1}

\def\R{\mathbb{R}}

\def\red{\color{red}}

\newtheorem{theorem}{Theorem}
\newtheorem{lemma}[theorem]{Lemma}
\newtheorem{proposition}[theorem]{Proposition}
\newtheorem{corollary}[theorem]{Corollary}

\DeclareMathSymbol{\varnothing}{\mathord}{AMSb}{"3F}

\begin{document}
\renewenvironment{proof}{\noindent {\bf Proof.}}{ \hfill\qed\\ }
\newenvironment{proofof}[1]{\noindent {\bf Proof of #1.}}{ \hfill\qed\\ }
\newenvironment{proofofof}[1]{\noindent {\bf Proof of #1.}}{}

\title{On the complexity  of polygonal billiards}  

\author{Tyll Krueger}
\address{Wrocław University of Technology, Faculty of Communication and Information Technology, ul.\ Janiszewskiego 11/17, 50-372 Wrocław, Poland}
\email{tyll.krueger@pwr.edu.pl}
\author{Arnaldo Nogueira}
\address{Aix Marseille Univ, CNRS,  I2M,  Marseille, France}
\email{arnaldo.nogueira@univ-amu.fr}

\author{Serge Troubetzkoy}
\address{Aix Marseille Univ, CNRS,  I2M,  Marseille, France}
\email{serge.troubetzkoy@univ-amu.fr}
 
\begin{abstract} 
We show that  the complexity of the billiard in a typical polygon grows 
cubically and the number of saddle connections grows quadratically along certain subsequences.
It is known that the set  of points whose first $n$-bounces hits the 
same sequence of sides as the orbit of an aperiodic phase point $z$
converges to 
$z$.
We establishe a polynomial lower bound estimate on this convergence rate for almost every $z$. This yields an upper bound on 
the upper metric complexity and upper slow entropy of
polygonal billiards. We also prove  significant deviations from the expected convergence behavior. Finally we extend these results to higher dimensions as well as to arbitrary invariant measures.
\end{abstract} 
\maketitle

\section{Introduction}
How complex is the billiard dynamics in a polygon?  According to Katok this is one of the two main open questions concerning arbitrary polygonal billiards \cite{Ka2} or even 
one of the five most resistant problems in dynamics \cite{Ka3}.
This question has been
the subject of many articles. 
The first such result, due to 
Boldrighini, Keane and Marchetti, and independently 
to Sinai, is that the metric entropy of the billiard 
map with respect to the Liouville measure (and thus of the billiard flow) is zero
\cite{BoKeMar,Si1,Si2}.
 Katok proved that the metric entropy of every invariant measure is zero as well \cite{Ka1}.  

To refine these results one can study two quantities.  A saddle connection is an orbit segment starting and ending at a corner of the polygon.  The quantity $N_g(t)$ denotes the number of saddle connections with geometric length at most $t$ and $N_c(n)$ denotes the number of saddle connections with combinatorial length at most $n$, 
here the number of links of a saddle connection is called its {\em combinatorial length}. 
Label the sides of the polygon by a finite alphabet
$\mathcal{A}$ and code the forward orbit of a point by the sequence of sides it hits.
The other quantity traditionally studied is $p(n)$ the number of distinct words of length $n$ which arise in this coding. 
Cassaigne, Hubert and Troubetzkoy showed that  
\begin{equation}\label{e1}
p(n) = \sum_{j=0}^{n-1} N_c(j)
\end{equation}
for all convex polygons \cite{CaHuTr}.
The proof given in \cite{CaHuTr} shows that this formula actually also holds for non-convex simply connected polygons since any $n$-cell is convex (\cite{Ka1}). A similar formula with possible low order correction terms was shown for all polygons by Bedaride \cite{Be}, correction terms are necessary for non simply connected polygons.

Masur showed that  $0 <  \liminf \frac{N_g(t)}{t^2}\le \limsup \frac{N_g(t)}{t^2} < \infty$ for any rational polygon  \cite{Mas1,Mas2}. Elementary arguments show that an analogous result holds for $N_c(n)$, thus the \cite{CaHuTr} relationship implies that  
$0 <  \liminf \frac{p(n)}{n^3} \le \limsup \frac{p(n)}{n^3}  < \infty$ for any rational polygon.

Suppose $Q$ is an arbitrary polygon, for two point $x,y \in Q$ let
$N_g(x,y,t)$ denote the number of paths of length at most $t$ connecting $x$ to $y$. In 1986, in an unpublished, but widely circulated letter to Masur, Boshernitzan showed that for fixed $\e > 0$ and $x \in Q$, for Lebesgue a.e.~$y$ we have $\lim \frac{N(x,y,t)}{t^{2+\e}} = 0$ and $\liminf \frac{N(x,y,t)}{t^{2+\e}} < \infty$, this result was published by Gutkin and Rams in 2009 \cite{GuRa}.
In this letter Boshernitzan also showed
that  there exist $Q,x,y$ such that $N(x,y,t) = 0$ for all $t > 0$, seemingly the first result about non-illumination for polygonal billiards.

Katok's result implies that each of  $N_g(t),N_c(n)$ and $p(n)$ grows sub-exponentially in any polygon, in particular the topological entropy of the billiard shift map $\sigma_Q: \Sigma_Q \to \Sigma_Q$ is zero, where 
$\Sigma_Q$ denotes the collection of all infinite codes obtained in the polygon $Q$.
Different proofs of zero topological entropy were given by Galperin, Krueger and Troubetzkoy \cite{GaKrTr}, as well as by Gutkin and Haydn \cite{GuHa1,GuHa2}.  

Scheglov has improved this for special classes of polygons, he first showed that for almost every triangle $N_c(n)$ grows slower than $e^{n^\varepsilon}$ for every $\varepsilon > 0$ \cite{Sc1,Sc2}, and then he showed that for almost every right triangle $N_c(n) \ge n^{\frac{2}{\sqrt{3}} - \e}$ infinitely often for every $\e > 0$ \cite{Sc3}.
For each $m \ge 3$ and each $k \ge 1$ Hooper 
exhibited an open set of $m$-gons for which 
$\liminf_{t \to \infty} N_g(t)/(t \log^k(t)) > 0$
\cite{Ho}.
Troubetzkoy showed that $p(n)$ grows at least quadratically in any polygon
\cite{Tr}.

In this article  we first study the  functions $N_g(t), N_c(n)$ and $p(n)$ for topologically typical polygons. For generic polygons, up to fluctuations the number or saddle connections grows quadratically, and the complexity of generic polygons grows cubically (Theorem \ref{t5}), other than Scheglov's right triangle result, these are the first polynomial growth result known for irrational billiards, and the exponents agree with the growth exponent of rational billiards.

Then we go on to study the metric version of the complexity 
or the slow entropy  of the billiard in an arbitrary polygon. 
We give an almost sure polynomial lower 
bound estimate on the decay of the size of an $n$-cell (Theorem \ref{t1})
and prove that this bound has  significant deviations (Theorem \ref{t2}). 
Theorem \ref{t1} implies that the upper metric complexity of the billiard shift in the sense of Ferenczi \cite{Fe}
and the upper slow entropy of the billiard shift in the sense of Katok and Thouvenot  \cite{KaTh} have polynomial upper bounds (Corollary \ref{c0}).

\section{Background material}
\subsection{Polygonal billiards}
Let $Q$ be a planar polygon.
The {\em billiard flow} $\phi_t$ in $Q$ is the free motion of a point 
particle with unit speed and specular reflections from the boundary; the 
reflections are not defined at corners of the polygon.  
The {\em phase space} \footnote{We use italic for the flow in contrast to $P_Q$ for the map.} $\mathcal{P}_Q$ of $\phi_t$ is the set of all inner pointing unit vectors, i.e.,
$$\mathcal{P}_Q = \{(x,\psi): x \in Q, \psi \in [0,2\pi), \text{ the vector } (x,\psi) \text{ points into Q }\}.$$

The {\em billiard map} $S$ is the first return map of the billiard flow to 
the boundary $\partial Q$, it is not defined for points for which the 
next arrival to the boundary is a vertex.
The {\em phase space} $P_Q$ of the billiard map is the set of inward pointing unit vectors with {\em foot point} in the boundary of $Q$ and in this case we will denote the foot point by $s$ (rather than $x$).  We will use the arc length parametrization of $\partial Q$.
The direction $\theta$ of a vector $z = (s,\theta)$ will refer to the angle the vector makes with the positive direction of the boundary. 
We  fix the positive orientation on $\partial Q$ and define
the angle at the vertices of $Q$ by one sided continuity.

We label the sides of an $m$-gon $Q$ by an $m$-element alphabet $\mathcal{A}$ and code each forward orbit  by the sequence of 
sides it hits.
For any $z \in P_Q$ such that $S^n(z)$ is defined the {\em $n$-cell} $C_n(z)$  is the open set of points for which $S^n$ is defined and whose coding sequence up to time $n$ as coincides with $z$'s coding sequence, equivalently
it is the largest open set containing $z$ for which 
the billiard map $S^n$ is continuous.  Katok 
showed that each $n$-cell is a convex polygon
whose boundary has at most $Kn$ sides where $K$ is an explicit constant, and each side 
of an $n$-cell is a branch of the singularity set of order $j$ where $0 \le j \le n$, 
i.e., points where $T^{j}$ is not defined but $T^{j-1}$ is defined \cite[Lemma 4]{Ka1}.

Let $\lambda$ and $\mu$ denote the  {\em invariant  Liouville measure} for the billiard flow, respectively, the billiard map, in our coordinates they are given by  $d\lambda = dx \, d\psi$ and  $d\mu = \sin \theta \, d\theta \, ds$.
We consider the $L_1$ metric $d((s,\theta),(s',\theta'))= |s-s'|+|\theta - \theta'|$
on ${P}_Q$
and denote the  $L_1$-ball in $P_Q$ centered at $z$ of radius $\varepsilon$ by
$B(z,\varepsilon)$.
We also consider the metric $\rho((x,\psi),(x',\psi')) = ||x-x'||_2 + |\psi - \psi'|$ on $\mathcal{P}_Q$, and notice that the two metrics agree
when restricted to a single side of the boundary.

To  define  a topology on the set of polygons with a fixed number of sides,
we define a basis of as follows.
Let $Q$ be a polygon and assume $\varepsilon$ is small enough such that 
the $\varepsilon$-neighborhoods  of the vertices of $Q$ are pairwise disjoint.
Then $\hat{Q}$ is in the open $\varepsilon$-neighborhood of $Q$
if  there is exactly one vertex of $\hat{Q}$ in the open
$\varepsilon$ neighborhood of each vertex of $Q$. 

\subsection{Polynomial entropy. }
The notion of topological slow entropy, or polynomial entropy, was first defined and studied by K\r{u}rka, Penne and Vaienti \cite{KuPeVa} in 2002 and independently by
Marco \cite{Mar1} in 2013. For a subshift  
the {\em lower} and {\em upper polynomial entropies} are defined as
\begin{equation*}
 \underline{h}_{poly}(\Sigma,\sigma) = \liminf_{n\to\infty} \frac{\ln p(n)}{\ln n},\;\;
    \overline{h}_{poly}(\Sigma,\sigma) = \limsup_{n\to\infty} \frac{\ln p(n)}{\ln n}.
\end{equation*} 
Here $p(n)$ denotes the number of distinct words of length $n$ in $\Sigma$. 

\subsubsection{Previous results on polynomial entropy and billiards}
Marco studied the upper polynomial entropy of billiards in smooth convex tables, he showed  that the upper polynomial entropy of the billiard map in a 
circle is 1, is 2 for non-circular ellipses, and at least 2 for any other smooth convex billiard table \cite{Mar2}. 

\subsection{Metric complexity or slow entropy}
We define the Ferenczi's metric complexity or Katok-Thouvenot's slow entropy in the setting of shift spaces \cite{Fe,KaTh}.  Let $\Sigma$ be a (one-sided) shift space over the alphabet $\mathcal{A}$ and $\sigma: \Sigma \to \Sigma$ the left shift map, which preserves an invariant measure $\nu$.

For two words $a,b \in \A^k$  we consider the $\bar d$ or Hamming distance 
$\bar{d}(a,b) := \frac{1}{k} \#\{i: a_i \ne b_i\}$.  For $x \in \Sigma$ let $\mathcal{B}(x,n,\e) := \{y \in \Sigma: \bar{d}(x_{0,n-1},y_{0,n-1)}) < \e\}$ where $x_{0,n-1}$ denotes the truncation of $x$ to its  first $n$
symbols. 
By \cite{GaKrTr} the partition into cylinders of length 1 (i.e., the sides of $Q$)
is a generating partition for the shift map $\Sigma_Q$, thus the definition
of metric complexity becomes somewhat simpler in this setting \cite[Cor 1]{Fe,KaTh}.
Let $P(n,\e)$ denote the smallest number $P$ such that there exists a subset of $\Sigma$ of measure at least $1-\e$ covered by  at most $P$ balls $\mathcal{B}(x,n,\e)$.
Let $g : \N \to \N$ be an increasing function.
We say $(\Sigma,\sigma,\nu)$ has {\em upper metric complexity} or {\em upper slow entropy} at most $g(n)$ if
$$\lim_{\e \to 0}\limsup_{n \to \infty} \frac{P(n,\e)}{g(n)} \le 1.$$

\section{Statements of results}
\begin{theorem}\label{t5}
For each $\ell \ge 3$ there is a dense $G_{\delta}$-set $G$ 
of polygons with $\ell$ sides  such that for each 
$Q \in G$  there exists an infinite strictly increasing 
integer sequence  $(m_k)$  such that
$$\lim_{k \to \infty} \frac{\log N_g^Q(m_k)}{\log m_k} = \lim_{k \to \infty} \frac{\log N_c^Q(m_k)}{\log m_k} = 2, \  \mbox{and} \ 
\lim_{k \to \infty} \frac{\log p^Q(m_k)}{\log m_k} = 3.$$
\end{theorem}

\begin{corollary}
For each 
$Q \in G$  
the polynomial entropies satisfy  
$$\underline{h}_{poly}(\Sigma_Q,\sigma_Q)  \le 3 \le \overline{h}_{poly}(\Sigma_Q,\sigma_Q).$$
\end{corollary}

We turn to the study of the metric complexity of polygonal billiards.
The set of points whose billiard orbit hits a given finite sequence of sides of length $n$ is called an $n$-cell, and $C_n(z)$ denotes the $n$-cell containing the point $z$.
The main result of \cite{GaKrTr} is that for each non-singular $z$
the set $C_n(z)$ converges to the point $z$ if and only if $z$ is aperiodic. 
Our first result in this direction is an almost sure lower 
bound estimate on the speed of this convergence. 
Let $P_Q$ denote the phase space of the billiard map and
 $B(z,\varepsilon)$ denotes the $L_1$-ball in $P_Q$ centered at 
$z$ of radius $\varepsilon$.
The next result and its corollary  improve implicit results of Sinai \cite{Si1,Si2}.

\begin{theorem}\label{t1} Let $Q$ be a polygon and $f: \mathbb{N} \to \mathbb{N}$ be a monotonically increasing function 
such that $\frac{1}{n f(n)}$ is summable.
Then there is a dense $G_\delta$-set $G\subset P_Q$ of full $\mu$ measure 
such that 
for  every $z \in G$  there exists $n_0 \ge 0$ such that
for all $n \ge n_0$ we have 
$$B\left (z,\frac{1}{n^3 f(n)} \right ) \subset C_n(z).$$
\end{theorem}

As previously mentioned 
the partition into cylinders of length 
is a generating partition, thus we get 
the following corollary.

\begin{corollary}\label{c0}
The upper metric complexity or upper slow entropy of the billiard shift is at most $n^6 (f(n))^2$ where
$f: \mathbb{N} \to \mathbb{N}$ is any monotonically increasing function 
such that $\frac{1}{n f(n)}$ is summable.
\end{corollary}

Our next result shows that the lower bound $1/(n^3f(n))$ in Theorem \ref{t1} can have rather large deviations.
\begin{theorem}\label{t2} Let $Q$ be a polygon and $f: \mathbb{N} \to \mathbb{N}$ be a monotonically increasing function.  Then 
there is a dense $G_\delta$-set $G \subset P_Q$ of full $\mu$ measure such that 
for  every $z \in G$ there exists an infinite strictly increasing sequence $n_i$  such that 
$$B\left (z,\frac{1}{n_i^2 f(n_i)} \right ) \subset C_{n_i}(z).$$
\end{theorem}

Theorems \ref{t1} and \ref{t2} have analogs for arbitrary invariant measures as well as for  higher dimensional polyhedra, these results will be stated in Section  \ref{s:6}. 

\section{Complexity of typical polygonal billiards}

\subsection{Estimates in small neighborhoods of a rational polygon}
We prove Theorem \ref{t5} by approximating irrational billiards by rational billiards.
To do this we need to give lower and upper bounds on 
how many saddle connections of length $n$ or how many $n$-cells 
can occur in a polygon close to a given rational polygon.  
Thus in this section we will emphasize the $Q$ dependence of 
the functions $N^Q_c(n)$ and $p^Q(n)$.

\begin{proposition}\label{p:1}
Suppose $Q$ is a rational polygon.  Then there exist 
constants  $0 < \underline{C} \le \overline{C}$ and $0< \underline{K} \le \overline{K}$ such that for each $n_0 \ge 1$
and each $\e \in (0,1)$ there
exists an open set $U$ whose closure contains $Q$ such that 
$$ \underline{C} n^2 \le N_c^{\hat{Q}}(n) \le (\overline{C}/\e) n^{2 + \e}
\quad \hbox{ and   } \quad 
\underline{K} n^3 \le p^{\hat{Q}}(n) \le (\overline{K}/\e) n^{ 3 + \e}$$ for all $\hat{Q} \in U$ and all $1 \le n \le n_0$.
\end{proposition}

We begin by a simple Lemma which proves the lower bound of the Proposition. 
\begin{lemma}\label{l:1}
Fix a polygon $Q$ and $n_0 \ge 1$, then there exists $\delta > 0$ such that  $N_c^{\hat{Q}}(n) \ge N_c^{Q}(n)$ and $p^{\hat{Q}}(n) \ge p^Q(n)$
for all  $1 \le n \le n_0$ and all 
$\hat{Q} \in B(Q,\delta)$.
\end{lemma}

\begin{proof}
The second inequality follows from the first inequality by \eqref{e1}.
Consider a saddle connection $\gamma$ of combinatorial length at most $n$.
Consider it in its unfolding.  Except for the starting and ending vertices
the distance of $\gamma$ to the other vertices is strictly positive. 
The unfolding varies continuously with the 
polygon in the sense that the vertices of the unfolded picture vary continuously, thus we can choose a small neighborhood of $Q$ for which $\gamma$ persists.  
The lemma follows since the collection of saddle connections of combinatorial length at most $n_0$ is finite.
\end{proof}

For the upper bound we need to work quite a bit more.
A {\em saddle chain} is a finite union of saddle connections which are on  
a straight line in an unfolding. A crucial point is
that under perturbation saddle chains can create new saddle connections.


\begin{lemma}\label{l:8}
Suppose $Q$ is a rational polygon with $\ell$ sides,  such that
$N^Q_c(n)\le \overline{C} n^2$ for all $n \ge 1$. Fix $n_0$ and $\e \in (0,1)$,  then there exists an open set of polygons $U$  containing $Q$ such that 
$$N_c^{\hat{Q}}(n) \le  \frac{9\overline{C}}{\e} n^{2 + \e} \text{ for all } \hat{Q} \in U
\text{ and } 1 \le n \le  n_0.$$
\end{lemma}

\begin{proof} 
Throughout the proof length will refer to 
combinatorial length. We need to understand how  new saddle connections can arise when perturbing $Q$. We claim the new saddle connections can only arise from a saddle chain.
To see this fix
a polygon $Q_0$ and a small enough neighborhood $U$ of $Q_0$ so that we
can identify sides and vertices of all the polygons in $U$. Consider 
a continuous one parameter family $\{Q_t \in U : t \in [0,1)\}$.
We fix a labelling of the vertices and edges of all the $Q_t$.
Suppose that there is  a saddle connection $\gamma^t$ connecting
a vertex $v_1$ to a vertex $v_2$ with code $w$
in all $Q_t$ for $t \in (0,1)$ but that this saddle connection
does not exist in $Q_0$.
We consider convergence in the unfolded picture, and use the same symbols for saddle connection in the unfolding and in the polygon.
Consider the unfoldings of $Q_t$ along $\gamma^t$. Both $\gamma^t$ and the associated unfoldings
vary continuously with $t$.{\footnote{More precisely, the unfolding can be viewed as a polygon in the plane, and the vertices of this polygon vary continuously with $t$.}} 
As $t\to 0$, the connections $\gamma^t$  converge to a straight line segment 
$\gamma^0$ connecting $v_1$ to $v_2$ in the unfolding of $Q_0$, since by assumption this is not a saddle connection there must be one or more vertices in between the two vertices, i.e., the saddle connection $\gamma^t$ has degenerated to a saddle chain $\gamma^0$, which proves our claim.
We remark that a whole side of the polygon can block the saddle connection, thus for the purpose of this proof a side of a polygon is considered to be a saddle connection, remember that this is not the case in
the Cassaigne, Hubert, Troubetzkoy  formula, this plays no role in our estimates since there are only finitely many such saddle connections or saddle chains.

We would like to choose the perturbation of the polygon  small enough that a saddle chain of length strictly less than $n_0$ does not create a new saddle chain of length at most $n_0$, however this turns out impossible in general, and we allow the creation
of certain special saddle chains as we will now explain.
Consider any saddle connection $\gamma$ of length $n < n_0$,
and let $\gamma^\pm$ denote its two forward continuations of length $n_0$.  
Note that either of these could be a saddle chain.
Fix one of these, say $\gamma^+$ and consider the associated unfolding.  
Let $V_{\gamma^+}$ denote the set of vertices of the unfolding contained in 
$\gamma^+$ and $V'$ denote the vertices of the unfolding not in $\gamma^+$
Consider the collection $SC(\gamma^+)$ of all saddle connections in the unfolding connecting vertices in $V_{\gamma^+}$ to vertices in $V'$.

Let $\Theta(Q,\gamma^+)$ be the minimal angle between the direction of
$\gamma^+$ and the directions of $SC(\gamma^+)$. 
Note that the $\gamma^+$ dependence of $\Theta$ is actually on dependence on  the unfolding of $\gamma^+$.
By assumption $\Theta(Q,\gamma^+) > 0$, furthermore as long a $\Theta$ remains strictly positive it varies continuously with the perturbation.
We require  $\delta$ to be so small that  for all $\hat Q \in U$ 
$\Theta(\hat{Q},\gamma^+) > \Theta(Q,\gamma^+)/2$, i.e., if any new saddle chain is created in only visits vertices from $V_{\gamma^+}$.
An example of such a new saddle chain is shown in the Figure \ref{fig1}, 
our estimates are worst case estimates for which we have counted this new saddle chain as having created all possible saddle connections. 
The final choice of $\delta$ must satisfy this requirement for each saddle connection of length at most $n_0$ and of course for $\gamma^-$ as well.


\begin{figure}   
\begin{minipage}[ht]{0.49\linewidth}
\centering
\begin{tikzpicture}[scale=1]
\draw[] (0,0.5) -- (0,1) -- (-0.2,1) -- (-0.2,2) -- (0,2) -- (0,2.5) -- (1,2.5) -- (1,2) -- (2,2) -- (2,1) -- (2.2,1) -- (2.2,0) -- (1,0) -- (1,0.5) -- (0,0.5);
\draw[thick, dotted, red] (1,0) -- (2,1) -- (1,2) -- (0,1);
\end{tikzpicture}
\end{minipage}\nolinebreak
\begin{minipage}[ht]{0.49\linewidth}
\centering
\begin{tikzpicture}[scale=1]
\draw[] (0,0.5) -- (0,1) -- (-0.2,1) -- (-0.2,2) --  (0,2) -- (0,2.5) -- (1,2.5) -- (1,2) -- (2,2) -- (2,0.8) -- (2.2,0.8) -- (2.2,0) -- (1,0) -- (1,0.5) -- (0,0.5) ;
\draw[thick, dotted, red] (1,0) -- (2,1) -- (1,2) -- (0,1); 
\end{tikzpicture}
\end{minipage}
\caption{After a small perturbation a saddle chain consisting of three saddle connections of  length 1 (left) can create a saddle chain consisting of a saddle connection of length two and a saddle connection of length one (right).
}\label{fig1}
\end{figure}
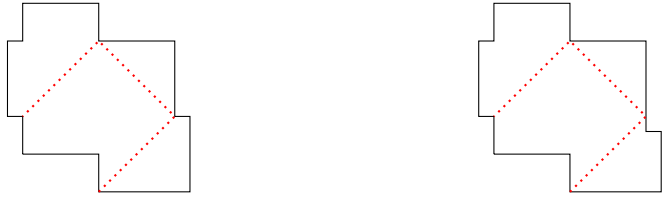

Thus to make our estimate we need to understand how many 
saddle chains there are and how many saddle connections a saddle chain can create under perturbation. Figure \ref{fig2}
depicts such a situation, a saddle chain consisting of three saddle connections of length 1 creates two additional saddle connections of length 2 and one additional saddle connection of length 3.



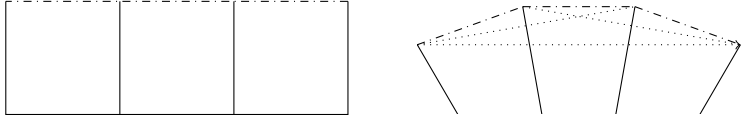
\begin{figure}
\def\deg{20}   
\begin{minipage}[ht]{.5\linewidth}
\centering
\begin{tikzpicture}[scale=1.5]
\draw[]  (2,1) -- (2,0) -- (1,0) -- (0,0) -- (-1,0) -- (-1,1);
\draw[dash dot] (2,1) -- (1,1) -- (0,1) -- (-1,1);
\draw[] (0,1) -- (0,0);
\draw[] (1,1) -- (1,0);
\end{tikzpicture}
\end{minipage}\nolinebreak
\begin{minipage}[ht]{.5\linewidth}
\begin{tikzpicture}[scale=4.25]
 \clip(-0.5,0.65) rectangle (.8,1);
\foreach \x in {3*\deg, 4*\deg, 5*\deg, 6*\deg}
{\draw[thin,-] (0:0)--(\x:1);}
\foreach \x in {4*\deg, 5*\deg, 6*\deg}
{\draw[dash dot, -] (\x:1) -- (\x-\deg:1);}
\foreach \x in {5*\deg, 6*\deg}
{\draw[dotted, -] (\x:1) -- (\x-2*\deg:1);}
\foreach \x in {6*\deg}
{\draw[dotted, ->] (\x:1) -- (\x-3*\deg:1);}
\draw[fill=white, white] (0,0) circle [radius=0.6cm];
\end{tikzpicture}
\end{minipage}
\caption{A saddle chain consisting of 3 saddle connections (dash-dotted on the left) can create 3 additional connections (dotted on the right).}\label{fig2}
\end{figure}

Now we turn to our estimate. 
The number of 
saddle chains in $Q$ of length
at most $n$ is bounded above by $2N^Q_c(n)$
since any saddle connection in $Q$ can be the start of at most 2 saddle chains.  

Consider a saddle chain consisting of $j$ saddle
connections in a polygon $Q$. Consecutive pairs of the $j+1$ endpoints are by 
definition connected by a saddle connection in $Q$.
Each non-consecutive pair of the $j+1$ endpoints 
of the $j$ saddle  connections can possibly create a new
saddle connection (such as in Figure \ref{fig2}).
Thus  a 
saddle chain consisting of $j$ saddle connections can create at most $\binom{j+1}{2} - j\le  j^2$ 
new saddle connections.

Fix $\e \in (0,1)$, $n_0 \ge 1$. Let $K = \lceil {2/\e} \rceil $ and consider $1 \le n \le n_0$.
We split the saddle chains of combinatorial length
at most $n$ into $K \ge 2$ classes according to the combinatorial length of the
shortest saddle connection in the saddle chain.
For each $1 \le i \leq K$ let $a_{i}:=\frac{i}{K}$. Then
let $I_1 := [0,a_1]$ and for $2 \le i \le K$ let $I_{i}:=\left( a_{i-1},a_{i}\right]$. A chain is in group $i$ if the shortest
generalized diagonal in the chain is of combinatorial length $n^{a}$ 
with $a\in I_{i}.$ For small $n$ some of these groups can be empty, but 
that does not effect our upper bound.

We would like to enumerate saddle chains by
the shortest saddle connection contained in them.
The starting point of the saddle chain  plays a role, for example there could be a saddle chain ending
with a saddle connection, and a different one starting with the same saddle connection.  Fix a saddle connection $\gamma$ of length 
$j <n$ which 
appears in a saddle chain of length at most $n$.  
Let $E_{\pm}(\gamma)$ be the two extensions by one sided continuity of $\gamma$ 
obtaining by extending $\gamma$
backwards and forwards by $n-j$. The length of $E_{\pm}(\gamma)$ is $2n-j$.
Each saddle chain of length $n$ containing $\gamma$ must be contained in $E_{+}(\gamma)$ or in $E_-(\gamma)$.
We consider the longest saddle chain $\hat{\gamma}_\pm$ contained in each of $E_{\pm}(\gamma)$,
the worst case estimate is that $\hat{\gamma}_\pm$ is of length $2n-j < 2n$.
Thus there are at most $2( N^Q_c(n^{a_i}) - N^Q_c(n^{a_{i-1}})) <  2 N^Q_c(n^{a_i}) \le 2 \overline{C} n^{2a_i}$
such saddle chains.

Suppose $\gamma$ is of length $n^a$ with $a \in I_i$. In the next paragraph we 
explain the estimates for $E_+(\gamma)$, the estimates for $E_{-}(\gamma)$ are similar.

There are two cases, the first case is when
$\gamma$ is the shortest saddle connection in {\bf all}  saddle chain of length $n$ containing it, so the shortest saddle connection in $E_+(\gamma)$. 
It follows that the longest saddle chain in $E_+(\gamma)$ can consist of at most 
$j= \lfloor 2n/n^{a_{i-1}} \rfloor$ saddle connections.
Applying the argument from above
yields that $E_+(\gamma)$
can create at most $j^2 = (2n/n^{a_{i-1}})^2$
new saddle connections.

Now suppose that $\gamma$ is the shortest saddle connection in some saddle chain of length $n$, but not the shortest in some other saddle chain of length $n$. 
In this case, we redefine $E_+(\gamma)$ by extending $\gamma$ only up to the first shorter saddle connection. Let $m < 2n$ be the length of the redefined $E_+(\gamma)$. Repeating the above arguments yields the longest saddle chain in $E_+(\gamma)$ can consist of at most 
$j= \lfloor m/n^{a_{i-1}} \rfloor < \lfloor 2n/n^{a_{i-1}} \rfloor$ saddle connections and
 $E_+(\gamma)$
can create at most $j^2 = (m/n^{a_{i-1}})^2 < (2n/n^{a_{i-1}})^2$
new saddle connections.

Combining these estimates yields 
that  the number of possible new generalized diagonals emerging from a chain in
group $i$ is bounded from above by 
$\overline{C} n^{2a_i} (2n/n^{a_{i-1}})^2 = 4\overline{C} n^{2 + 2a_i - 2a_{{i-1}}} = 4\overline{C} n^{2 + 2/K}$.
Since we have $K$ groups the total number of newly
generated generalized diagonals emerging from chains of length $n$ is
bounded by  $4K \overline{C}  n^{2 + 2/K} =  4\lceil {2/\e} \rceil \overline{C}  n^{2 + 2/K} \le   (9 \overline{C}/\e)  n^{2 + \e}
$.
\end{proof}

\begin{proofof}{Proposition \ref{p:1}}
The lower bounds follow from Lemma \ref{l:1} by choosing $\underline{C}$ and $\underline{K} = \underline{C}/2$ the corresponding constants for the rational polygon $Q$.

On the other hand the upper bound follows from Lemma \ref{l:8} combining with the formaula 
$p(n) = \sum_{j=0}^{n-1} N_c(j)$ from
\cite{CaHuTr} to $\hat Q \in U$ yields
\begin{eqnarray*}
p^{\hat Q}(n+1)  
& \le & N_c(0) + \frac{8 \overline{C}}{\e}   \sum_{m=1}^{n} m^{2 + \e}
 \le   N_c(0) + \frac{8 \overline{C}}{\e} \int_1^{n+1} x^{2+\e} \, dx\\
& \le  & N_c(0) +  \frac{8 \overline{C}}{\e(3+\e)} (n+1)^{3+\e}.
\end{eqnarray*}
Thus we can choose a strictly positive constant $\overline{K}$
such that 
$p^{\hat Q}(n)  \le (\overline{K}/\e) n^{3+\e}$
for all $1 \le n \le n_0$.
\end{proofof}

\subsection{}
\begin{proofofof}{Theorem \ref{t5}}
Fix $\ell \ge 3$.
Consider a dense set $\{Q_i\}$ of rational polygons with $\ell$ sides. 
For short we will write $N_c^i(n)$ for $N_c^{Q_i}$ and $p^i(n)$
for $p^{Q_i}(n)$.
For each $Q_i$ we consider the constants $\underline{C_i},\underline{K_i},\overline{C_i},\overline{K_i}$ given by
Proposition \ref{p:1}, applied to  $\e_i := 1/i$ and $n_i$ so large that  
\begin{equation}\label{e2}  \max \left ( \left |\frac{\log(\underline{C_i})}{\log n_i} \right |,\left |\frac{\log(\underline{K_i})}{\log n_i} \right |, \left | \frac{\log(i\overline{C_i})}{\log n_i} \right |, \left | \frac{\log(i\overline{K_i})}{\log n_i} \right |
\right )
< \frac{1}{i}.
\end{equation}
According to Proposition \ref{p:1} we can choose an open set
 $U_i$  for $Q_i,\e_i,n_i$, such that 
for all $Q \in U_i$ we have 
\begin{equation*}\label{est}
\begin{split}
\frac{\log(N_c^Q(n_i))}{\log n_i}  & 
\ge \frac{\log(\underline{C_i}n_i^2)}{\log n_i} = 2 + \frac{\log(\underline{C_i})}{\log n_i}
> 2 - \frac{1}{i} \\
\frac{\log(p^Q(n_i))}{\log n_i} & 
\ge \frac{\log(\underline{K_i}n_i^3)}{\log n_i} = 3 + \frac{\log(\underline{K_i})}{\log n_i}
> 3 - \frac{1}
{i} 
\end{split}
\end{equation*}
and
\begin{equation*}
\begin{split}
\frac{\log(N_c^Q(n_i))}{\log(n_i)} & \le \frac{\log(i\overline{C_i} n_i^2)}{\log(n_i)}  = {2} + \frac{\log(i\overline{C_i} )}{\log(n_i)} 
< {2} + \frac{1}{i}.\\
\frac{\log(p^Q(n_i))}{\log(n_i)} & \le \frac{\log(i\overline{K_i} n_i^3)}{\log(n_i)}  = 3 + \frac{\log(i\overline{K_i} )}{\log(n_i)} 
< 3 + \frac{1}{i}.
\end{split}
\end{equation*}

Now consider the dense $G_\delta$-set
$$G := \bigcap_{k \ge 2} \bigcup_{i \ge k} U_i.
$$
For each $Q \in G$ there is an increasing sequence $i_k$ such that 
\begin{equation*}
\begin{split}
2 - \frac{1}{i_k} & < \frac{\log(N_c^Q(n_{i_k}))}{\log(n_{i_k})} 
< {2} + \frac{1}{i_k}\\
3 - \frac{1}{i_k} & < \frac{\log(p^Q(n_{i_k}))}{\log(n_{i_k})} 
< 3 + \frac{1}{i_k}
\end{split}
\end{equation*}
hold for each $i_k$. Setting $m_k := n_{i_k}$  yields
$\lim_{k \to \infty} \frac{\log N_c^Q(m_k)}{\log m_k} = 2$ and 
$\lim_{k \to \infty} \frac{\log p^Q(m_k)}{\log m_k} = 3$.

The assertion, 
$\lim_{k \to \infty} \frac{\log N_g^Q(m_k)}{\log m_k} = \lim_{k \to \infty} \frac{\log N_c^Q(m_k)}{\log m_k}$,  follows after a slight modification of the proof. For each polygon $Q$
there are positive constants $L_Q < d(Q) := \rm{diameter}(Q)$ and $N_0$ such that if $\gamma$ is a saddle connection of combinatorial length $n \ge N_0$  then its geometric length is 
contained in the interval $[L_Q n, d(Q) n]$.  Thus implies that $N_g(L_Q n) \le N_c(n) \le N_g(d(Q) n)$.
The stated result is scale invariant, so we can assume $d(Q)=1$, thus we
obtain $N_c(n) \le N_g(n) \le N_c(\lceil n/L_Q \rceil )$.
We immediately conclude 
$$\lim_{k \to \infty} \frac{\log N_g^Q(m_k)}{\log m_k} \ge \lim_{k \to \infty} \frac{\log N_c^Q(m_k)}{\log m_k} = 2.$$

To prove the reverse inequality we will modify the above construction so
that not only $\frac{\log N_c^Q(m_k)}{\log m_k} = 2$ but also $\frac{\log N_c^Q(\lceil m_k/L_Q \rceil)}{\log m_k} = 2$, which will finish the proof.

To prove the reverse inequality note that 
the constant $L_Q$ depends weakly on $Q$, it can be chosen 
constant in a neighborhood $\mathcal{N}$ of $Q$. In the rest of the proof we fix
the neighborhood and the constant $L := L_Q$.

Now in the previous construction in \eqref{e2} we additionally require that if $Q_i \in \mathcal{N}$ then $n_i$ satisfies
$$  \left |  \frac{\log(i\overline{C_i}/L^2)}{\log  n_i} \right |
< \frac{1}{i}$$
and applying  Proposition \ref{p:1} to  $Q_i,\e_i,\lceil n_i/L \rceil$.
Then
\begin{equation*}
\begin{split}
\frac{\log N_g^Q(m_k)}{\log m_k} & \le  \frac{\log N_c^Q(\lceil m_k/L \rceil)}{\log m_k} \le \frac{\log(i\overline{C_i} (n_i/L)^2)}{\log(n_i)}\\ & = {2} + \frac{\log(i\overline{C_i}/L^2 )}{\log(n_i)} 
\le {2} - \frac{1}{i}. \hspace{1.75in} \qed 
\end{split}
\end{equation*}
\end{proofofof}

To improve Theorem \ref{t5} from a subsequence result to a result for all $n$, one would need to have effective bounds on the various constants and the rate of approximation, this seems difficult.

\section{Metric complexity of polygonal billiards}
We begin by a geometric result, we will show that points whose orbit stays far from the boundary of the phase space must have a large $n$-cell. The boundary 
$\partial \mathcal{P}_Q$ of the phase space $\mathcal{P}_Q$ consists
of vectors tangent to a side of the polygon $Q$ and vectors whose
foot point is a vertex of $Q$.
If $Q$ is not convex then the flow orbit of a point $z \in P_Q$ can pass close to a vertex 
with neither $z$ nor $S(z)$ being close to a vertex (Figure \ref{fig3}), this motivates the 
following definition. 
For $z \in P_Q$ let $t(z)$ denote the geometric length of the orbit segment between $z$ and $S(z)$. 
Let
\begin{equation}\label{e:Ba}B_{a} := \{z  \in P_Q: \phi_t(z) \in B(\partial \mathcal{P}_Q,a)
\text{ for some } t \in [0,t(z)]\}
\end{equation}
and  for $T > 0$ let 
$$G_a^T := \{z \in {P}_Q: \phi_t(z) \not \in B_{a} 
\text{ for } 0 \le t \le T\}.$$
Let $T_n := n \cdot \text{diam}(Q)$.

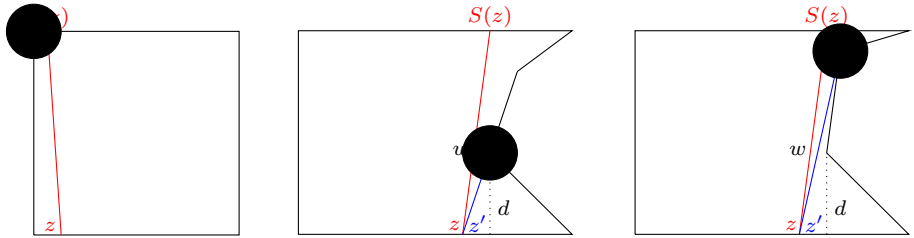
\begin{figure}[ht]
\begin{minipage}[ht]{.3\linewidth}
    \centering
    \begin{tikzpicture}[scale=1.8]
       \draw[] (1,0) -- (2.5,0) -- (2.5,1.5) -- (1,1.5) -- (1,0) ;
           \node at (1.12,0.062) {\red \tiny $z$};
            \node at (1.1,1.6) {\red \tiny $S(z)$};
        \draw[red] (1.2,0) -- (1.1,1.5);
         \draw[fill, opacity=0.1] (1,1.5) circle (0.2);
    \end{tikzpicture}
    \end{minipage}\nolinebreak
\begin{minipage}[ht]{.35\linewidth}
    \centering
    \begin{tikzpicture}[scale=1.8]
       \draw[] (1,0) -- (3,0) -- (2.4,0.6) -- (2.6,1.2) -- (3,1.5) -- (1,1.5) -- (1,0) ;
       \draw[dotted] (2.4,0) -- (2.4,0.6);
        \node at (2.5,0.2) {\tiny $d$};
           \node at (2.14,0.062) {\red \tiny $z$};
           \node at (2.31,0.08) {\color{blue} \tiny $z'$};
            \node at (2.4,1.6) {\red \tiny $S(z)$};
            \node at (2.5,0.6) {\tiny $v$};
               \node at (2.19,0.62) {\tiny $w$};
            \draw[green] (2.4,0.6) -- (2.278,0.614);
        \draw[red] (2.2,0) -- (2.4,1.5);
        \draw[blue] (2.2,0) -- (2.4,0.6);
         \draw[fill, opacity=0.1] (2.4,0.6) circle (0.2);
    \end{tikzpicture}
    \end{minipage}\nolinebreak
    \begin{minipage}[ht]{.35\linewidth}
    \centering
    \begin{tikzpicture}[scale=1.8]
       \draw[] (1,0) -- (3,0) -- (2.4,0.6) -- (2.5,1.35) -- (3,1.5) -- (1,1.5) -- (1,0) ;
       \draw[dotted] (2.4,0) -- (2.4,0.6);
        \node at (2.5,0.2) {\tiny $d$};
           \node at (2.14,0.062) {\red \tiny $z$};
           \node at (2.31,0.08) {\color{blue} \tiny $z'$};
            \node at (2.4,1.6) {\red \tiny $S(z)$};
            \node at (2.55,1.25) {\tiny $v$};
               \node at (2.19,0.62) {\tiny $w$};
            \draw[green] (2.5,1.35) -- (2.38,1.36);
        \draw[red] (2.2,0) -- (2.4,1.5);
        \draw[blue] (2.2,0) -- (2.5,1.35);
         \draw[fill, opacity=0.1] (2.5,1.35) circle (0.2);
    \end{tikzpicture}
    \end{minipage}
    \caption{Three examples of $z \in B_a$.}
    \label{fig3}
\end{figure}

\begin{proposition}\label{p:9} For each polygon $Q$, $a>0$, $n \ge 1$, and  $z \in G_a^{T_n}$ we have
 \begin{equation*}B(z,a/(2 + T_n)) 
\subset C_{n}(z)
\end{equation*}
\end{proposition}
\begin{proof}
There is a strictly positive constant $k_2$ 
depending only on $Q$ such that any orbit segment $\phi_{[0,T]}(z)$ makes at most $k_2 T$ collisions, thus
\begin{equation}\label{e3}
P_Q \setminus G_a^T = \{z \in P_Q: \phi_t(z) \in B_{a} \text{ for some } 0 \le t \le T\} \subset 
\hspace{-0.3cm}
\bigcup_{0 \le j \le k_2 T} S^{-j} B_{a}.
\end{equation}

Suppose $z \in G_a^T$. For calculational 
convenience we consider
 a flow-based version of $n$-cells, let $\mathcal{C}_T(z) \subset P_Q$ be the maximal 
neighborhood of $z$ (in $P_Q$)  such that $\phi_T|_{\mathcal{C}_T}$ is continuous.
A billiard orbit segment of geometric length $\text{diam}(Q)$ must make at least one bounce. 
It follows that  a billiard orbit segment of geometric length $T$ must make at least $T/\text{diam}(Q)$
bounces, and thus 
$$\mathcal{C}_T(z) \subset C_{\lfloor T/\text{diam}(Q) \rfloor}(z).$$
Note that
$C_{\lfloor k_2 T \rfloor }(z) \subset \mathcal{C}_T(z)$, but we will not use this.

Our goal is to estimate the largest $\delta > 0$ such that
$B(z,\delta) \subset \mathcal{C}_T(z)$;  $\delta$ is always assumed small enough so that  this inclusion holds.

Call a subset of parallel vectors in $B(z,\delta)$ a 
{\em  horizontal segment}. By assumption $\phi_T$ is continuous on $B(z,\delta)$, and thus when restricted
to a horizontal segment it is an isometry.

We call a  subset of vectors in $B(z,\delta)$ with the same base point in $B(z,\delta)$
a {\em vertical segment}.
Now we will consider the action of $\phi_T$ on vertical segments. Let $\ell \subset B(z,\delta)$ be  the vertical segment with lower endpoint $(s,\theta_0)$ 
and upper endpoint $(s,\theta_1)$
where $2\delta = \theta_1 - \theta_0$.
For  $r \in [0, 1]$ let 
$\theta_r := (\theta_1 - \theta_{0}) r+ \theta_{0}$, 
so $\ell = \{z_r := (s,\theta_r) : 0 \le r \le 1\}$. 
In the unfolding of the orbit segment $z$ to $\phi_T(z)$, the image $\phi_T (\ell)=:  
\{z_{r,T} := (\bar{x}_{r,T},{\theta_r}): 0 \le r \le 1\}$ is an arc of angle $\delta r$ of a circle of radius $T$ (Figure \ref{fig4}), thus in the unfolding we trivially have the following bound 
$||\bar{x}_{0,T} -\bar{x}_{r,T}|| <T \delta r$ for each $0 \le r \le 1$. On the other hand $\theta_r - \theta_0 = \delta r$.    
Thus in the phase space $\mathcal{P}_Q$
we have $\rho(z_{0,T}, z_{r,T}) \le T \delta r + \delta r = (T+1)\delta r$ for each $0 \le r \le 1$.

\begin{figure}
    \centering
    \begin{tikzpicture}
       \draw[] (0,0) -- (2,0);
 \draw[] (1,0) -- (1.5,4);
 \draw[] (1,0) --(2.5,3.87);
\draw (1.5,4) arc (90:75:3.91) ;
 \node at (0.8,0.15) {\tiny $z_0$};
  \node at (1.25,0.13) {\tiny $z_1$};
  \node at (1.25,4) {\tiny  $z_{0,T}$};
   \node at (2.8,3.8) {\tiny  $z_{1,T}$};
    \end{tikzpicture}
    \caption{Unfolding a vertical segment}
    \label{fig4}
\end{figure}
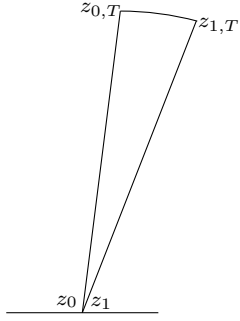

Now we will combine the above horizontal and vertical estimates.
For $z_1 = (s_1,\theta_1), z_2 = (s_2,\theta_2) \in B(z,\delta)$
we consider $z_3 = (s_1,\theta_2)$.
The estimates of the previous two paragraphs and
the triangle inequality yield
\begin{eqnarray*}
\rho(\phi_T(z_1),\phi_T(z_2)) & \le & \rho(\phi_T(z_1),\phi_T(z_3)) + \rho(\phi_T(z_3),\phi_T(z_2))\\
& \le & \delta + (T+1)\delta = (2 + T)\delta.
\end{eqnarray*}

We have arrived at the punch line of the geometric estimate.
If $z \in G_a^T$, we have $\rho(\phi_t(z),\partial P) \ge a$
for all $0 \le t \le T$.
Thus if $\delta$ satisfies
$\delta (2 + T)  < a$ or equivalently
$\delta < a/(2 + T)$ then all points in $B(z,\delta)$
belong to $\mathcal{C}_T(z)$ and thus we have $B(z,a/(2 + T)) 
\subset C_{\lfloor T/\text{diam}(Q) \rfloor }(z)$.
Applying this to $T_n := n \cdot \text{diam}(Q)$ yields the desired result.
\end{proof}

\begin{proofof}{Theorems \ref{t1} and \ref{t2}}
The first step of the proof is to estimate the quantity $\mu(B_a)$.
If $Q$ is convex we have 
\begin{equation}\label{e4}
B_{a} \subset B(\partial P_Q,a) \cup S^{-1} B(\partial P_Q,a)
\end{equation}
and thus there is a constant $k_1 >0$ such that
\begin{equation}\label{e4.1}
    \mu(B_{a}) \le k_1 a.
\end{equation}
Such an estimate remains true in the non-convex case, although \eqref{e4} does not hold.

To see that \eqref{e4.1} holds  we need to consider the measure of points in
$z=(s,\theta) \in B_a\setminus (B(\partial P_Q,a) \cup S^{-1} B(\partial P_Q,a))$; such a point is depicted in red in the central and right parts of  Figure \ref{fig3}.
Let $\mathcal{S}$ consist of the points whose $S$-image is not defined, i.e., hits a vertex, for example
the blue segment in Figure \ref{fig3} starts in $\mathcal{S}$.
Let
$d$ be the smallest distance of a non-convex vertex\footnote{i.e., the segment connecting
two close enough points on the boundary of $Q$ on opposite sides of the vertex lies outside $Q$.}  to a side not containing this vertex (dotted segment in the figure).
Among the points closest to the segment $\phi_{[0,t(z)]}(z)$ there must be a vertex.  There could
 be several of them, then we choose the last one to which the flow arrives, call
it $v$. Then there is a direction
$\theta'$ so that $z' = (s,\theta') \in \mathcal{S}$, this is denoted in blue in the figure. It has length at least $d$.
Consider the (green) segment starting at $v$ and perpendicular to the segment  $\phi_{[0,t(z)]}(z)$, call its other end point $w$. By assumption the segment $[v,w]$ it has length less that $a$.
This right triangle with vertices $z,v,w$ yields 
$|\theta - \theta'| < \arcsin \frac{a}{d} < \frac{a}{d}$, and thus 
$z \in B(\mathcal{S}, \frac{a}{d})$. The $d$ in this estimate could be
improved in certain cases, for example in the right part of Figure \ref{fig3}, but this is not necessary for us.
We have shown
\begin{equation}\label{e5}
B_{a} \subset B(\partial P_Q,a) \cup B(\mathcal{S}, a/d) \cup S^{-1} B(\partial P_Q,a).
\end{equation}
Since  the set $\mathcal{S}$ is a smooth curve in the phase space, we again obtain a (different) constant $k_1 >0$ such that $\mu(B_{a}) \le k_1 a$. 

This inequality combined with \eqref{e3} and the fact that  $\mu$ is $S$-invariant implies
we can choose a $k_3$ such that 
\begin{equation}\label{e6}\mu(P_Q \setminus G_a^T)  \le (k_2 T + 1) \mu(B_a) = (k_2 T + 1) k_1 a
\le k_3 T a.\end{equation}

Now we have all the estimates necessary to make the  $G_\delta$ argument. Let
$K := 2 +\text{diam}(Q)$. 
For Theorem \ref{t1} let $a_n := \frac{K}{n^2 f(n)}$ while for Theorem \ref{t2} let $a_n := \frac{K}{n  f(n)}$ and consider the corresponding sets $B_{a_n}$.
Applying Proposition \ref{p:9} to $z \in G_{a_n}^{T_n}$ yields $B(z,1/(n^3 f(n))) 
\subset C_{n}(z)$, resp.\ $B(z,1/(n^2 f(n))) \subset C_{n}(z)$.

For both choices of $(a_n)$ Equation \eqref{e6}  yields
$\mu(G_{a_n}^{T_n})  \to 1$ as $n \to \infty$.
Let 
$$G := \bigcap_{n_0 \ge 1} \bigcup_{n \ge n_0} G_{a_n}^{T_n}.$$
Notice that $G$ is a dense $G_\delta$-set of full measure.
If $z \in G$ then there is an infinite sequence $n_i$
such that $z \in G_{a_{n_i},T_{n_i}}$ for all $i$ which finishes the proof of Theorem \ref{t2}.
To prove Theorem \ref {t1} we notice that $\mu(G_{a_n}^{T_n})$ is summable
and we apply the Borel-Cantelli lemma to the sets $G_{a_n}^{T_n}$ to conclude that for a.e. $z\in G$ there is an $n_0$ such that for all $n\ge n_0$ we have $z \in G_{a_n}^{T_n}$.
\end{proofof}

\begin{proofof}{Corollary \ref{c0}}
Any notation in $\Sigma_Q$ which corresponds to a notation in $Q$ will be denoted with a bar,  for example $\bar \mu$ denote the lift  of the invariant measure $\mu$ to $\Sigma$ and $\bar z \in \Sigma_Q$ denote the code of $z \in P_Q$.  For a point $z$ in the set $G$ from Theorem \ref{t1}, the Theorem  yields an $n_0$ such that $\bar \mu(\bar z_{[0,n-1]}) = \mu(C_n(z)) \ge \frac{const}{n^6(f(n))^2} $ for all $n \ge n_0$. Let $\bar z[0,n-1]$ denote the cylinder set
corresponding to the word $\bar z_{0,n-1}$, here $const$ is the normalizing constant of the measure $\mu$.
Let $G_N := \{z\in G : n_0(z) \le N\} $. Since $G = \cup_{N \ge 1} G_N$ is of full measure for each $\e > 0$ we can
choose $N$ such that $\mu(G_N) \ge 1 - \e$, thus its lift 
satisfies $\bar\mu(\bar G_N) > 1 - \e$.
Since distinct cylinder sets are disjoint we obtain a cover of $\bar{G}_N$ by cylinders of measure at least $\frac{const}{n^6f(n)^2} $, so there are at most $\frac{(1-\e) n^6 f(n)^2}{const}$ such cylinders.
For each $z$ the cylinder set $\bar z_{[0,n-1]}$ is just the ball $\mathcal{B}(\bar z,n,0)$.  
Any cover by cylinders ($0$ balls) is a cover by $\e$ balls, thus the metric complexity can only be smaller than our estimate. 
\end{proofof}

\section{Generalizations of metric complexity results}\label{s:6}
\subsection{Other invariant measures}
Both  Theorems \ref{t1} and \ref{t2} have versions for an arbitrary invariant measure on $P_Q$.
We remind the reader that 
the set $B_a$ was defined in \eqref{e:Ba}.

\begin{theorem}\label{t10} Let $Q$ be a polygon, $\nu$
an $S$-invariant measure,  and $f: \mathbb{N} \to \mathbb{N}$ be a monotonically  increasing  function 
such that $\frac{1}{n f(n)}$ is summable.
There exists a non-increasing sequence $(a_n)$ satisfying $\nu(B_{a_n})< 1 /n^2 f(n)$. For any such sequence there
exists a set $G$ of full $\nu$-measure which is a dense $G_\delta$-subset of $\text{supp}(\nu)$ with the property that 
for  every $z \in G$  there exists $n_0 \ge 0$ such that
for all $n \ge n_0$ we have 
$$B\left (z,\frac{a_n}{2 + n \cdot \text{diam}(Q)} \right ) \subset C_n(z).$$
\end{theorem}

For the version of Theorem \ref{t2} we get:

\begin{theorem}\label{t11} Let $Q$ be a polygon, $\nu$
an $S$-invariant measure, and $f: \mathbb{N} \to \mathbb{N}$ be a monotonically increasing unbounded function. 
There is  a  non-increasing sequence $(a_n)$  
which satisfies $\nu(B_{a_n}) < K_0 / n f(n)$.  For any such sequence
there is  a set $G$ of full $\nu$-measure which is a dense $G_\delta$-subset of $\text{supp}(\nu)$ with the property that
for  every $z \in G$ there exists an infinite strictly increasing sequence $n_i$  such that 
$$B\left (z,\frac{a_{n_i}}{2 + {n_i} \cdot \text{diam}(Q)} \right ) \subset C_{n_i}(z).$$
\end{theorem}

\begin{proofof}{Theorems \ref{t10} and \ref{t11}} The proof is a minor modification of 
the proof of Theorems \ref{t1} and \ref{t2}. 
We begin by showing  the existence of the sequences $(a_n)$.
Inclusion \eqref{e5} yields
\begin{equation}\label{eextra}\nu(B_a) \le 2 \nu(B(\partial Q,a)) + \nu(B(\mathcal{S},a)).
\end{equation}
By definition $\nu$ is supported on the set of 
non-singular points, in particular $\nu(\partial {P}_Q) = 0$ and $\nu(\mathcal{S}) = 0$.
Thus we can choose $(a_n)$ as in the statements of the two theorems since the right hand side of \eqref{eextra} goes to zero as $a \to 0$.  

The sequences $(a_n)$ are implicitly define, they differ from those in Theorems \ref{t1} and \ref{t2}, but by definition the estimates of measures of the sets are the same
as for $\mu$.
Thus for both Theorems we have  $\nu(P_Q \setminus G_n^{T_n}) \to 0$ and for Theorem \ref{t10}
we can apply the Borel-Cantelli Lemma.  The  rest of the proof follows from Proposition \ref{p:9} in an identical fashion, the only difference being that we do not
know the exact form of the sequence $(a_n)$.
\end{proofof}

For irrational polygons we know very little about invariant measures. We give some sufficient conditions for which we can understand the growth of $a_n/(2 +n \cdot \text{diam}(Q))$.

If there exists $a >0$ such that  $\nu(\{z\in P_Q: d(z,\partial {P}_Q) < a\}) = 0$
for sufficiently small $a$, then  Theorem \ref{t10}
tells us that for almost every $z$ there exists $n_0 \ge 0$ such that
for all $n \ge n_0$ we have 
$$B\left (z,\frac{a}{2 + n \cdot \text{diam}(Q)} \right ) \subset C_n(z).$$
This occurs for example for any invariant measure on a collection of periodic orbits which stay away from the boundary.

Suppose now that $\nu(B(\partial {P}_Q,a)) \le a^r$ for some $r > 0$ and all sufficiently small $a$. In this case, the conclusion of Theorem \ref{t10} reads for $\nu$-a.e.\ $z$ there exists $n_0$ such that for all $n \ge n_0$ we have
$$B\left (z,\frac{1}{\sqrt[r]{n^2 f(n)}(2 + n \cdot \text{diam}(Q))} \right ) \subset C_{n}(z)$$
while the conclusion of Theorem \ref{t11}
reads for $\nu$-a.e.\ $z$ there exists an infinite strictly increasing sequence $n_i$ such that 
$$B\left (z,\frac{1}{\sqrt[r]{n_i f(n_i)}(2n_i+1)} \right ) \subset C_{n_i}(z).$$
Of course the Liouville measure satisfies this with $r=1$.
Another example would be to take the uniform measure on a periodic cylinder.  
The boundary of a cylinder must touch a vertex somewhere, thus this also yields an example with $r=1$
in any polygon with a periodic orbit.

\subsection{Higher dimensions} Our techniques work for polyhedra in any dimension. 
In the $m$-dimensional case the phase space $P_Q$ of the billiard map is 
$2(m-1)$-dimensional ($m-1)$ spatial coordinates and $(m-1)$ angular coordinates.  
The boundary of the phase space is then $(2m-3)$-dimensional, coming from inner 
pointing vectors whose foot point is in the intersections of sides of the polyhedra, and the set of points tangent to a face.
If again $\mu$ denotes the Liouville measure we obtain
$\mu(B_{a_n}) \le C a_n^{2m-3}$ and 
$\mu(P_Q \setminus G_n) \le  
n C a_n^{2m-3}$.  The estimate in Proposition \ref{p:9} on linear separation does not
change, thus
for a version of Theorem \ref{t1} we choose 
$a_n = \frac{K}{\sqrt[2m-3]{n^2 f(n)}}$ while for a 
version of  Theorem \ref{t2} we use 
$a_n = \frac{K}{\sqrt[2m-3]{n f(n)}}$ which yields the following two results.

\begin{theorem}\label{t12} Let $Q$ be a polyhedron in $\R^m$  and $f: \mathbb{N} \to \mathbb{N}$ be a monotonically  increasing function 
such that $\frac{1}{n f(n)}$ is summable.
Then there is a dense $G_\delta$-set $G$ of full $\mu$-measure such that 
for  every $z \in G$  there exists $n_0 \ge 0$ such that
for all $n \ge n_0$ we have 
$$B\left (z,\frac{1}{n \sqrt[2m-3]{n^2 f(n)}} \right ) \subset C_n(z).$$
\end{theorem}

\begin{theorem}\label{t13} Let $Q$ be a polyhedron in $\R^m$ and $f: \mathbb{N} \to \mathbb{N}$ be a monotonically increasing function.  Then there is a dense $G_\delta$-set $G$ of full $\mu$-measure such that 
for  every $z \in G$ there exists an infinite strictly increasing sequence $n_i$  such that 
$$B\left (z,\frac{1}{n_i \sqrt[2m-3]{n_i f(n_i)}} \right ) \subset C_{n_i}(z).$$
\end{theorem}

\end{document}